\documentclass[11pt]{article}
\usepackage{latexsym,amssymb,amsmath,amsfonts,amsthm}
\usepackage{graphics,graphicx,mathrsfs,subfigure}
\usepackage{color,xspace}
\newcommand{\be}{\begin{eqnarray}}
\newcommand{\ben}{\begin{eqnarray*}}
\newcommand{\en}{\end{eqnarray}}
\newcommand{\enn}{\end{eqnarray*}}

\newtheorem{theorem}{Theorem}[section]
\newtheorem{lemma}[theorem]{Lemma}

\newtheorem{remark}[theorem]{Remark}

\definecolor{rot}{rgb}{1,0,0}
\definecolor{hw}{rgb}{0,0,1}

\begin{document}
\renewcommand{\theequation}{\arabic{section}.\arabic{equation}}
\title{\bf

A refinement of the Bukhgeim-Klibanov method
 
}
\author{ Suliang Si\thanks{School of Mathematics and Statistics, Shandong University of Technology,
Zibo, 255000, China ({\tt sisuliang@amss.ac.cn})}}
\date{}

 



\maketitle

\begin{abstract}
 In this article, we improve the classical Bukhgeim-Klibanov method presented in \cite{BK1981}, which can be used to prove the conditional stability of inverse source problem for a hyperbolic equation from the measurement on the subboundary. A major ingredient of our proof is a novel Carleman estimate. This inequality eliminates the need to extend the solution in time, therefore simplifies the existing proofs, which is widely applicable to various evolution equations.
 
\end{abstract}

%


\section{Introduction}

Let $\Omega\subset\mathbb{R}^n$ be a  bounded domain with $C^2$-boundary $\partial\Omega$, $n \geq 1$. We consider the wave equation
\begin{equation}\label{u}
\begin{cases}
\partial_t^2 u(x,t) - \Delta u(x,t) + q(x) u(x,t) = 0 & (x,t)\in \Omega \times (0, T), \\
u(x,0) = u_0(x), \quad \partial_t u(x,0) = u_1(x) & x\in \Omega,\\
u(x,t) = h(x,t) & (x,t)\in \partial\Omega \times (0, T), \\
\end{cases}
\end{equation}
Here $u$ denotes the amplitude of the waves, $q$ is a potential,
$h$ is a boundary term  and $(u_0, u_1)$ are the initial data.
 
Define, for $m > 0$, the space
\[
L_{\leq m}^{\infty}(\Omega) = \bigl\{ q \in L^{\infty}(\Omega),\ \|q\|_{L^{\infty}(\Omega)} \leq m \bigr.\}.
\] 
It is known that initial boundary value problem (\ref{u}) is well-posed. Indeed, if we assume $(u_0, u_1) \in H^2(\Omega)\times H^1(\Omega)$, $h\in H^2(0,T;H^2(\Omega))$ and $q\in L_{\leq m}^{\infty}(\Omega)$,  there exists a unique solution \(u=u_q\) to  (\ref{u}) such that 
\[u \in C^1([0,T];H^2(\Omega)) \cap C^2([0,T];H^1(\Omega)).\]
Assume there exists $ x_0 \notin \overline{\Omega}$, such that $\Gamma_0 \supset \{ x \in \partial\Omega,\ (x - x_0) \cdot \nu(x) \geq 0 \}$
and
\[T_0= \sup_{x\in\Omega} |x - x_0|.\]
In (\ref{u}), assume that \(u_0\), \(u_1\) and \(h\) are given, the inverse problem is to estimate $q_1 - q_2$ in a suitable norm by $\partial_\nu u_{q_1}(x,t) - \partial_\nu u_{q_2}(x,t)$ on $\Gamma_0\times(0,T)$.
This question has all already received positive answers.
The following theorem is a classical result in the theory of  inverse coefficient problems.
\begin{theorem}\label{Thm}
Let $T > T_0$ and $(u_0, u_1) \in H^2(\Omega)\times H^1(\Omega)$ , $h\in H^2(0,T;H^2(\Omega))$. We assume that there exist constant $M_0,\  m_0 > 0$ such that
\begin{equation}
|u_0(x)| \geq m_0 > 0, \quad x \in \Omega, \quad \text{and} \quad \|u_{q_2}\|_{H^1(0,T;L^\infty(\Omega)} \leq M_0.
\end{equation}
Then there exists a constant $C > 0$ such that
\[
\|q_1 - q_2\|_{L^2(\Omega)} \leq C \left\| \partial_\nu \partial_t u_{q_1} - \partial_\nu \partial_t u_{q_2} \right\|_{L^2(\Gamma_0\times(0,T))}, \quad q_1, q_2 \in L_{\leq m}^{\infty}(\Omega).
\]
Here the constant $C$ is dependent on $\Omega$, $T$, $m_0$, $M_0$.
\end{theorem}

The inverse coefficient problem under consideration
has been well-studied in the literature, starting with the uniqueness result in the
celebrated article \cite{BK1981}, 
Bukhgeim and Klibanov  proposed a fundamental method based on
 a global Carleman estimate. Thus Carleman estimates became a fundamental
tool for establishing uniqueness  for inverse coefficient problems.

 Later on, stability issues were obtained for the wave equation, first based
on the observability properties of the wave equation \cite{PY96, PY97} and then refined
with the use of Carleman estimates, among which are \cite{IY01a, IY01b, IY03, KY06}.

In fact, a great part of the literature in this area, concerning uniqueness, stability,
and reconstruction of coefficient inverse problems for evolution partial differential
equations, can be found in \cite{Beilina2012, Bellassoued2017, BBE17, Fu2019, HIY, Imanuvilov1998, Imanuvilov2001, Imanuvilov2001a, Imanuvilov2003, JLY2017,K2002,Y1995,Y1999,PY96,PY97,IY01a,IY01b,IY03,KY06}, and we refer the interested
reader to it. 

In this paper, our main innovation is to present a new proof of the theorem. Existing methods all require extending the solution to \((-T,T)\) followed by integration by parts. Here, we perform integration by parts directly on \((0,T)\).

\section{Carleman estimates}
Let \[\psi(x,t) = |x - x_0|^2 - \beta t^2 + \beta_0, \quad \beta\in(0,1)\]
where $\beta_0 > 0$ is chosen such that $\psi \geq 1$ in $\Omega \times (0, T)$.
We define the weight function
\[\varphi(x,t)=e^{\lambda \psi(x,t)},\]
where \(\lambda>0\) is a second large parameter. 
Throughout this article, \(M>0\) denote generic constants which are independent of
parameter \(s\).

The Carleman estimate is a fundamental
tool for establishing uniqueness and stability for inverse problems. 
A Carleman estimate is an $L^2$-weighted estimate for the wave operator, and is stated
as follows:
\begin{lemma}\label{Car}
Let $v(x, 0) = 0$ for all $x \in \Omega$.
Then there exist $s_0 > 0$, $\lambda > 0$ and a positive constant $M$ such that for all $s \geq s_0$:
\begin{equation}\label{ine}
\begin{aligned}
&s^{1/2} \int_\Omega e^{2s\varphi(0)} |\partial_t v(0)|^2 dx + s \int_{0}^T \int_\Omega e^{2s\varphi} \left( |\partial_t v|^2 + |\nabla v|^2 \right) dxdt + s^3 \int_{0}^T \int_\Omega e^{2s\varphi} |v|^2 dxdt \\
&\leq M \int_{0}^T \int_\Omega e^{2s\varphi} |\partial_t^2v-\Delta v+qv|^2 dxdt + M s \int_{0}^T \int_{\Gamma_0} e^{2s\varphi} |\partial_\nu v|^2 d\sigma dt\\
&+Ms  \int_\Omega e^{2s\varphi(T)} \left( |\partial_t v(T)|^2 + |\nabla v(T)|^2 \right) dx + s^3 \int_\Omega e^{2s\varphi(T)} |v(T)|^2 dx,
\end{aligned}
\end{equation}
for all $v \in C^1((0,T); H^1(\Omega))$ satisfying $(\partial_t^2-\Delta) v \in L^2(\Omega \times (0,T))$, $\partial_\nu v \in L^2(\partial\Omega \times (0,T))$.
\end{lemma}
Lemma \ref{Car} is a classical Carleman estimate.  The main step of the proof is to first take the even extension of \(u\) to \((-T, T)\), followed by integration by parts in \(\Omega\times(-T,T)\). 
We now present a new proof. We take integration by parts directly in \(\Omega\times(0,T)\), without extending \(u\) of (\ref{u}) to \((-T, T)\). we only need \(u(x,0)=0\) for \(x\in\Omega\), not \(\partial_tu(x,0)=0\), \(x\in\Omega\).
If \(T\) is large enough, the terms of Lemma \ref{u} at times \(t=T\)
 can be removed, which can be found in \cite{BBE13}, we omit it.
\begin{remark}
Let $v(x, 0) = 0$ for all $x \in \Omega$. If \(T > \sup\limits_{x\in\Omega} |x - x_0|\), 
then there exist $s_0 > 0$, $\lambda > 0$ and a positive constant $M$ such that for all $s \geq s_0$:
\begin{equation}
\begin{aligned}
&s^{1/2} \int_\Omega e^{2s\varphi(0)} |\partial_t v(0)|^2 dx + s \int_{0}^T \int_\Omega e^{2s\varphi} \left( |\partial_t v|^2 + |\nabla v|^2 \right) dxdt + s^3 \int_{0}^T \int_\Omega e^{2s\varphi} |v|^2 dxdt \\
&\leq M \int_{0}^T \int_\Omega e^{2s\varphi} |\partial_t^2v-\Delta v+qv|^2 dxdt + M s \int_{0}^T \int_{\Gamma_0} e^{2s\varphi} |\partial_\nu v|^2 d\sigma dt
\end{aligned}
\end{equation}
for all $v \in C^1((0,T); H^1(\Omega))$ satisfying $(\partial_t^2-\Delta) v \in L^2(\Omega \times (0,T))$, $\partial_\nu v \in L^2(\partial\Omega \times (0,T))$.
More generally, we have
\begin{equation}
\begin{aligned}
s^{1/2} \int_\Omega e^{2s\varphi(0)} |\partial_t v(0)|^2 dx 
\leq M \int_{0}^T \int_\Omega e^{2s\varphi} |\partial_t^2v-\Delta v+qv|^2 dxdt + M s \int_{0}^T \int_{\Gamma_0} e^{2s\varphi} |\partial_\nu v|^2 d\sigma dt
\end{aligned}
\end{equation}
\end{remark}

\begin{proof}

It is sufficient to prove Lemma \ref{Car} in the case where \(q\equiv 0\). Indeed, we assume that we already established
the inequality 
\begin{equation}\label{ine2}
\begin{aligned}
&s^{1/2} \int_\Omega e^{2s\varphi(0)} |\partial_t v(0)|^2 dx + s \int_{0}^T \int_\Omega e^{2s\varphi} \left( |\partial_t v|^2 + |\nabla v|^2 \right) dxdt + s^3 \int_{0}^T \int_\Omega e^{2s\varphi} |v|^2 dxdt \\
&\leq M \int_{0}^T \int_\Omega e^{2s\varphi} |\partial_t^2v-\Delta v|^2 dxdt + M s \int_{0}^T \int_{\Gamma_0} e^{2s\varphi} |\partial_\nu v|^2 d\sigma dt\\
&+Ms  \int_\Omega e^{2s\varphi(T)} \left( |\partial_t v(T)|^2 + |\nabla v(T)|^2 \right) dx + s^3 \int_\Omega e^{2s\varphi(T)} |v(T)|^2 dx.
\end{aligned}
\end{equation}
Since \(q\in L_{\leq m}^{\infty}(\Omega)\), we have
\begin{equation}
|\partial_t^2v-\Delta v|^2\leq |\partial_t^2v-\Delta v+qv-qv|^2\leq 2 |\partial_t^2v-\Delta v|^2+2m^2|v|^2.
\end{equation}
By choosing \(s\) large, we can absorb the term
\[\int_{0}^T \int_\Omega e^{2s\varphi} |v|^2 dxdt\]
into the left-hand side of the Carleman estimate (\ref{ine2}) with \(q=0\).

In order to prove the Carleman estimates, we set \[w=e^{s\varphi}v \quad \mbox{for all } (x,t)\in\Omega\times(0,T).\] 
Then, we introduce the conjugate operator $P$ defined by
\begin{equation}
Pw = e^{s\varphi} (\partial_t^2-\Delta) \left( e^{-s\varphi} w \right).
\end{equation}
Some easy computations give
\begin{equation}
\begin{aligned}
Pw &= \partial_t^2 w - 2s\lambda\varphi\left( \partial_t w \partial_t\psi - \nabla w \cdot \nabla\psi \right) + s^2\lambda^2\varphi^2 w\left( |\partial_t\psi|^2 - |\nabla\psi|^2 \right) - \Delta w \\
&\quad - s\lambda\varphi w\left( \partial_t^2\psi - \Delta\psi \right) - s\lambda^2\varphi w\left( |\partial_t\psi|^2 - |\nabla\psi|^2 \right)-s\partial_t w \\
&= P_1 w + P_2 w + R_1w+R_2w,
\end{aligned}
\end{equation}
where
\begin{align}
P_1 w=\partial_t^2 w- \Delta w+ s^2\lambda^2\varphi^2 w\left( |\partial_t\psi|^2 - |\nabla\psi|^2 \right),
\end{align}
\begin{align}
P_2 w &= (\alpha - 1)s\lambda\varphi w(\partial_t^2\psi - \Delta\psi) - s\lambda^2\varphi w(|\partial_t\psi|^2 - |\nabla\psi|^2) \notag \\
&\quad - 2s\lambda\varphi(\partial_t w\partial_t\psi - \nabla w\cdot\nabla\psi)-s\partial_t w, 
\end{align}
and
\begin{align}
R_1 w &= -\alpha s\lambda\varphi w(\partial_t^2\psi - \Delta\psi), \qquad
R_2 w=-s\partial_t w.
\end{align}
Let
\begin{equation}
\alpha \in \left( \frac{2\beta}{\beta + n}, \frac{2}{\beta + n} \right).
\end{equation}
Since we have
\begin{equation}
\int_{0}^T \int_\Omega \left( |P_1 w|^2 + |P_2 w|^2 \right) dxdt + 2 \int_{0}^T \int_\Omega P_1 w P_2 w dxdt = \int_{0}^T \int_\Omega |P w - R_1 w-R_2w|^2 dxdt, 
\end{equation}
the main part of the proof is then to bound from below the cross-term
\[
\int_{0}^T \int_\Omega P_1 w P_2 w dxdt=\sum_{k=1}^9 I_k.
\]
We calculate the six terms \(I_k\), \(k=1,2\cdot\cdot\cdot 9\) by integrating by parts with respect to  \((x,t)\).

Integrations by part in time give easily
\[
\begin{aligned}
I_{1} &= \int_{0}^T \int_\Omega \partial_t^2 w \left( (\alpha - 1)s\lambda\varphi w (\partial_t^2\psi - \Delta\psi) \right) dxdt \\
&=(\alpha - 1)s\lambda\int_{\Omega}\partial_t w w\varphi (\partial_t^2\psi - \Delta\psi) (T)dx-\frac{(\alpha - 1)s\lambda}{2}\int_{\Omega} |w|^2\partial_t\varphi (\partial_t^2\psi - \Delta\psi) (T)dx\\
&= (1 - \alpha)s\lambda \int_{0}^T \int_\Omega \varphi |\partial_t w|^2 (\partial_t^2\psi - \Delta\psi) dxdt \\
&\quad - \frac{(1 - \alpha)}{2} s\lambda^2 \int_{0}^T \int_\Omega \varphi |w|^2 \partial_t^2\psi (\partial_t^2\psi - \Delta\psi) dxdt \\
&\quad - \frac{(1 - \alpha)}{2} s\lambda^3 \int_{0}^T \int_\Omega \varphi |w|^2 |\partial_t\psi|^2 (\partial_t^2\psi - \Delta\psi) dxdt.
\end{aligned}
\]
Similarly, one has
\[
\begin{aligned}
I_{2} &= \int_{0}^T \int_\Omega \partial_t^2 w \left( -s\lambda^2\varphi w \left( |\partial_t\psi|^2 - |\nabla\psi|^2 \right) \right) dxdt \\
&=-s\lambda^2\int_\Omega \partial_t w  w\varphi(|\partial_t\psi|^2 - |\nabla\psi|^2)(T)dx+\frac{s\lambda^2}{2}\int_\Omega |w|^2\partial_t \left(\varphi(|\partial_t\psi|^2 - |\nabla\psi|^2)\right)(T)dx\\
&= s\lambda^2 \int_{0}^T \int_\Omega \varphi |\partial_t w|^2 \left( |\partial_t\psi|^2 - |\nabla\psi|^2 \right) dxdt - s\lambda^2 \int_{0}^T \int_\Omega \varphi |w|^2 |\partial_t^2\psi|^2 dxdt \\
&\quad - \left( 2 + \frac{1}{2} \right) s\lambda^3 \int_{0}^T \int_\Omega \varphi |w|^2 |\partial_t\psi|^2 \partial_t^2\psi dxdt + \frac{s\lambda^3}{2} \int_{0}^T \int_\Omega \varphi |w|^2 |\nabla\psi|^2 \partial_t^2\psi dxdt \\
&\quad - \frac{s\lambda^4}{2} \int_{0}^T \int_\Omega \varphi |w|^2 |\partial_t\psi|^2 \left( |\partial_t\psi|^2 - |\nabla\psi|^2 \right) dxdt
\end{aligned}
\]
and
\[
\begin{aligned}
I_{3} &= \int_{0}^T \int_\Omega \partial_t^2 w \left( -2s\lambda\varphi \left( \partial_t w \partial_t\psi - \nabla w \cdot \nabla\psi \right) \right) dxdt \\
&= -s\lambda \int_\Omega |\partial_tw|^2\varphi\partial_t\psi(T)dx+2s\lambda \int_\Omega \partial_tw \varphi\nabla w\cdot\nabla\psi(T)dx\\
&+ s\lambda \int_{0}^T \int_\Omega \varphi |\partial_t w|^2 \partial_t^2\psi \, dxdt + s\lambda^2 \int_{0}^T \int_\Omega \varphi |\partial_t w|^2 |\partial_t\psi|^2 \, dxdt \\
&\quad + s\lambda \int_{0}^T \int_\Omega \varphi |\partial_t w|^2 \Delta\psi \, dxdt + s\lambda^2 \int_{0}^T \int_\Omega \varphi |\partial_t w|^2 |\nabla\psi|^2 \, dxdt \\
&\quad - 2s\lambda^2 \int_{0}^T \int_\Omega \varphi \partial_t w \partial_t\psi \nabla w \cdot \nabla\psi \, dxdt.
\end{aligned}
\]
Furthermore, by Green’s formula and integration by parts, we obtain
\[
\begin{aligned}
I_{4} &= \int_{0}^T \int_\Omega -\Delta w \left( (\alpha - 1)s\lambda\varphi w (\partial_t^2\psi - \Delta\psi) \right) dxdt \\
&= -(1 - \alpha)s\lambda \int_{0}^T \int_\Omega \varphi |\nabla w|^2 (\partial_t^2\psi - \Delta\psi) dxdt \\
&\quad + \frac{(1 - \alpha)}{2} s\lambda^2 \int_{0}^T \int_\Omega \varphi |w|^2 \Delta\psi (\partial_t^2\psi - \Delta\psi) dxdt \\
&\quad + \frac{(1 - \alpha)}{2} s\lambda^3 \int_{0}^T \int_\Omega \varphi |w|^2 |\nabla\psi|^2 (\partial_t^2\psi - \Delta\psi) dxdt.
\end{aligned}
\]
On the other hand,
\[
\begin{aligned}
I_{5} &= \int_{0}^T \int_\Omega -\Delta w \left( -s\lambda^2\varphi w \left( |\partial_t\psi|^2 - |\nabla\psi|^2 \right) \right) dxdt \\
&= -s\lambda^2 \int_{0}^T \int_\Omega \varphi |\nabla w|^2 \left( |\partial_t\psi|^2 - |\nabla\psi|^2 \right) dxdt \\
&\quad - \frac{s\lambda^2}{2} \int_{0}^T \int_\Omega \varphi |w|^2 \Delta\left( |\nabla\psi|^2 \right) dxdt \\
&\quad + \frac{s\lambda^3}{2} \int_{0}^T \int_\Omega \varphi |w|^2 \Delta\psi \left( |\partial_t\psi|^2 - |\nabla\psi|^2 \right) dxdt \\
&\quad + \frac{s\lambda^4}{2} \int_{0}^T \int_\Omega \varphi |w|^2 |\nabla\psi|^2 \left( |\partial_t\psi|^2 - |\nabla\psi|^2 \right) dxdt \\
&\quad - s\lambda^3 \int_{0}^T \int_\Omega \varphi |w|^2 \nabla\psi \cdot \nabla\left( |\nabla\psi|^2 \right) dxdt.
\end{aligned}
\]
Using the fact that $w|_{\partial\Omega\times(0,T)} = 0$, $\nabla w = (\partial_\nu w)\nu$ and $|\nabla w|^2 = |\partial_\nu w|^2$ on $\partial\Omega\times(0,T)$, we obtain
\[
\begin{aligned}
I_{6} &= \int_{0}^T \int_\Omega -\Delta w \left( -2s\lambda\varphi \left( \partial_t w \partial_t\psi - \nabla w \cdot \nabla\psi \right) \right) dxdt \\
&=-s\lambda \int_\Omega|\nabla w|^2\varphi\partial_t\psi(T)dx\\
&+ s\lambda \int_{0}^T \int_\Omega \varphi |\nabla w|^2 (\partial_t^2\psi - \Delta\psi) dxdt + 2s\lambda^2 \int_{0}^T \int_\Omega \varphi |\nabla\psi \cdot \nabla w|^2 dxdt \\
&\quad - 2s\lambda^2 \int_{0}^T \int_\Omega \varphi \partial_t w \partial_t\psi \nabla w \cdot \nabla\psi dxdt + s\lambda^2 \int_{0}^T \int_\Omega \varphi |\nabla w|^2 \left( |\partial_t\psi|^2 - |\nabla\psi|^2 \right) dxdt \\
&\quad - s\lambda \int_{0}^T \int_{\partial\Omega} \varphi |\partial_\nu w|^2 \nabla\psi \cdot \nu \, d\sigma dt + 4s\lambda \int_{0}^T \int_\Omega \varphi |\nabla w|^2 dxdt.
\end{aligned}
\]
One easily writes
\[
\begin{aligned}
I_{7} &= \int_{0}^T \int_\Omega s^2\lambda^2\varphi^2 w \left( |\partial_t\psi|^2 - |\nabla\psi|^2 \right) \left( (\alpha - 1)s\lambda\varphi w (\partial_t^2\psi - \Delta\psi) \right) dxdt \\
&= (\alpha - 1)s^3\lambda^3 \int_{0}^T \int_\Omega \varphi^3 |w|^2 (\partial_t^2\psi - \Delta\psi) \left( |\partial_t\psi|^2 - |\nabla\psi|^2 \right) dxdt
\end{aligned}
\]
and
\[
\begin{aligned}
I_{8} &= \int_{0}^T \int_\Omega s^2\lambda^2\varphi^2 w \left( |\partial_t\psi|^2 - |\nabla\psi|^2 \right) \left( -s\lambda^2\varphi w \left( |\partial_t\psi|^2 - |\nabla\psi|^2 \right) \right) dxdt \\
&= -s^3\lambda^4 \int_{0}^T \int_\Omega \varphi^3 |w|^2 \left( |\partial_t\psi|^2 - |\nabla\psi|^2 \right)^2 dxdt.
\end{aligned}
\]
Finally, some integrations by part enable to obtain
\[
\begin{aligned}
I_{9} &= \int_{0}^T \int_\Omega s^2\lambda^2\varphi^2 w \left( |\partial_t\psi|^2 - |\nabla\psi|^2 \right) \left( -2s\lambda\varphi \left( \partial_t w \partial_t\psi - \nabla w \cdot \nabla\psi \right) \right) dxdt \\
&=-s^3\lambda^3\int_\Omega |w|^2\varphi^3 (|\partial_t\psi|^2 - |\nabla\psi|^2)\partial_t\psi(T)dx\\
&+ s^3\lambda^3 \int_{0}^T \int_\Omega \varphi^3 |w|^2 (\partial_t^2\psi - \Delta\psi) \left( |\partial_t\psi|^2 - |\nabla\psi|^2 \right) dxdt \\
&\quad + 2s^3\lambda^3 \int_{0}^T \int_\Omega \varphi^3 |w|^2 \left( \partial_t^2\psi |\partial_t\psi|^2 + 2|\nabla\psi|^2 \right) dxdt \\
&\quad + 3s^3\lambda^4 \int_{0}^T \int_\Omega \varphi^3 |w|^2 \left( |\partial_t\psi|^2 - |\nabla\psi|^2 \right)^2 dxdt
\end{aligned}
\]
and
\[
\begin{aligned}
I_{10}&=\int_{0}^T \int_\Omega P_1 w (-s^{1/2}\partial_t w) \, dxdt\\
&= \int_{0}^T \int_\Omega \left( \partial_t^2 w - \Delta w + s^2\lambda^2\varphi^2 w \left( |\partial_t\psi|^2 - |\nabla\psi|^2 \right) \right)(-s^{1/2} \partial_t w) \, dxdt \\
&= \frac{s^{1/2}}{2} \int_\Omega |\partial_t w(0)|^2 dx -\frac{s^{1/2}}{2} \int_\Omega |\partial_t w(T)|^2 dx \\
&+\frac{s^{1/2}}{2}\int_{\Omega}|\nabla w|^2(T)dx\\
&+\frac{s^{5/2}\lambda^2}{2} \int_{0}^T \int_\Omega |w|^2  \left( \varphi^2 \left( |\partial_t\psi|^2 - |\nabla\psi|^2 \right) \right)(T) dxdt\\
&+ \frac{s^{5/2}\lambda^2}{2} \int_{0}^T \int_\Omega |w|^2 \partial_t \left( \varphi^2 \left( |\partial_t\psi|^2 - |\nabla\psi|^2 \right) \right) dxdt.
\end{aligned}
\]
Gathering all the terms that have been computed, we get
\begin{equation}
\begin{aligned}
\int_{0}^T \int_\Omega P_1 w P_2 w \, dxdt
&= 2s\lambda \int_{-T}^T \int_\Omega \varphi |\partial_t w|^2 \partial_t^2\psi \, dxdt - \alpha s\lambda \int_{-T}^T \int_\Omega \varphi |\partial_t w|^2 (\partial_t^2\psi - \Delta\psi) dxdt \\
&\quad + 2s\lambda^2 \int_{-T}^T \int_\Omega \varphi \left( |\partial_t w|^2 |\partial_t\psi|^2 - 2\partial_t w \partial_t\psi \nabla w \cdot \nabla\psi + |\nabla\psi \cdot \nabla w|^2 \right) dxdt \\
&\quad + 4s\lambda \int_{-T}^T \int_\Omega \varphi |\nabla w|^2 dxdt + \alpha s\lambda \int_{-T}^T \int_\Omega \varphi |\nabla w|^2 (\partial_t^2\psi - \Delta\psi) dxdt \\
&\quad - s\lambda \int_{-T}^T \int_{\partial\Omega} \varphi |\partial_\nu w|^2 \nabla\psi \cdot \nu(x) \, d\sigma dt \\
&\quad + 2s^3\lambda^4 \int_{-T}^T \int_\Omega \varphi^3 |w|^2 \left( |\partial_t\psi|^2 - |\nabla\psi|^2 \right)^2 dxdt \\
&\quad + 2s^3\lambda^3 \int_{-T}^T \int_\Omega \varphi^3 |w|^2 \left( \partial_t^2\psi |\partial_t\psi|^2 + 2|\nabla\psi|^2 \right) dxdt \\
&\quad + \alpha s^3\lambda^3 \int_{-T}^T \int_\Omega \varphi^3 |w|^2 (\partial_t^2\psi - \Delta\psi) \left( |\partial_t\psi|^2 - |\nabla\psi|^2 \right) dxdt + X_1\\
&- \frac{s^2\lambda^2}{2} \int_{0}^T \int_\Omega |w|^2 \partial_t \left( \varphi^2 \left( |\partial_t^2\psi|^2 - |\nabla\psi|^2 \right) \right) dxdt\\
&+(\alpha - 1)s\lambda\int_{\Omega}\partial_t w w\varphi (\partial_t^2\psi - \Delta\psi) (T)dx-\frac{(\alpha - 1)s\lambda}{2}\int_{\Omega} |w|^2\partial_t\varphi (\partial_t^2\psi - \Delta\psi) (T)dx\\
&-s\lambda^2\int_\Omega \partial_t w  w\varphi(|\partial_t\psi|^2 - |\nabla\psi|^2)(T)dx+\frac{s\lambda^2}{2}\int_\Omega |w|^2\partial_t \left(\varphi(|\partial_t\psi|^2 - |\nabla\psi|^2)\right)(T)dx\\
&-s\lambda \int_\Omega |\partial_tw|^2\varphi\partial_t\psi(T)dx+2s\lambda \int_\Omega \partial_tw \varphi\nabla w\cdot\nabla\psi(T)dx\\
&-s\lambda \int_\Omega|\nabla w|^2\varphi\partial_t\psi(T)dx
-s^3\lambda^3\int_\Omega |w|^2\varphi^3 (|\partial_t\psi|^2 - |\nabla\psi|^2)\partial_t\psi(T)dx\\
&+\frac{s^{1/2}}{2} \int_\Omega |\partial_t w(0)|^2 dx -\frac{s^{1/2}}{2} \int_\Omega |\partial_t w(T)|^2 dx +\frac{s^{1/2}}{2}\int_{\Omega}|\nabla w|^2(T)dx\\
&-\frac{s^{5/2}\lambda^2}{2} \int_{0}^T \int_\Omega |w|^2  \left( \varphi^2 \left( |\partial_t^2\psi|^2 - |\nabla\psi|^2 \right) \right)(T) dxdt\\
\end{aligned}
\end{equation}
where $X_1$ gathers the non-dominating terms and satisfies
\[
|X_1| \leq M s \lambda^4 \int_{0}^{T} \int_{\Omega} \varphi |w|^2 \, dx dt.
\]
More details can be found in \cite{BBE13} and are omitted here. Consequently we have
\begin{equation}\label{pp}
\begin{aligned}
&\int_{0}^T \int_\Omega P_1 w P_2 w \, dxdt + 2s\lambda \int_{-T}^T \int_{\partial\Omega} \varphi |\partial_\nu w|^2 (x - x_0) \cdot \nu(x) \, d\sigma dt + |X_1| \\
&+M s\lambda \int_\Omega \varphi \left( |\partial_t w|^2 + |\nabla w|^2 \right)(T) dx + M s^3\lambda^3  \int_\Omega \varphi^3 |w|^2(T) dx\\
&\geq M s\lambda \int_{-T}^T \int_\Omega \varphi \left( |\partial_t w|^2 + |\nabla w|^2 \right) dxdt + M s^3\lambda^3 \int_{-T}^T \int_\Omega \varphi^3 |w|^2 dxdt+\frac{s^{1/2}}{2} \int_\Omega |\partial_t w(0)|^2 dx.
\end{aligned}
\end{equation}
Since 
\begin{equation}\label{pr}
\begin{aligned}
\int_{0}^T \int_\Omega |Pw - R_1w-R_2w|^2 dxdt
&\leq 3 \int_{0}^T \int_\Omega |Pw|^2 dxdt + 3 \int_{0}^T \int_\Omega |R_1w|^2 dxdt+ \int_{0}^T \int_\Omega |R_2w|^2 dxdt\\
&\leq M \int_{-T}^T \int_\Omega |Pw|^2 dxdt + M s^2\lambda^2 \int_{0}^T \int_\Omega \varphi^2 |w|^2 dxdt\\
&+Ms\int_{0}^T \int_\Omega  |\partial_tw|^2 dxdt,
\end{aligned}
\end{equation}
using (\ref{pp}) and (\ref{pr}), we get
\[
\begin{aligned}
&s\lambda \int_{0}^T \int_\Omega \left( |\partial_t w|^2 + |\nabla w|^2 \right) \varphi \, dxdt + s^3\lambda^3 \int_{0}^T \int_\Omega |w|^2 \varphi^3 \, dxdt + \frac{s^{1/2}}{2} \int_\Omega |\partial_t w(0)|^2 dx \\
&\leq M \int_{0}^T \int_\Omega |P w|^2 dxdt + M s\lambda \int_{-T}^T \int_{\Gamma_0} \varphi |\partial_\nu w|^2 (x - x_0) \cdot \nu(x) \, d\sigma dt \\
&\quad +M s\lambda \int_\Omega \varphi \left( |\partial_t w|^2 + |\nabla w|^2 \right)(T) dx + M s^3\lambda^3  \int_\Omega \varphi^3 |w|^2(T) dx.
\end{aligned}
\]
Thus the proof of Lemma \ref{Car} is complete.

\end{proof}

\section{Proof of Theorem \ref{Thm}}

Let \(z=u_{q_1}-u_{q_2}\)， we obtain
\begin{equation}\label{z}
\begin{cases}
\partial_t^2 z - \Delta z + q_1(x) z = (q_2-q_1)u_{q_2} & \text{in } \Omega \times (0, T), \\
z(0) = 0, \quad \partial_t z(0) = 0 & \text{in } \Omega,\\
z = 0 & \text{on } \partial\Omega \times (0, T). \\
\end{cases}
\end{equation}
Furthermore, set \(v=\partial_tz\), then \(v\) satisfies
\begin{equation}\label{v}
\begin{cases}
\partial_t^2 v - \Delta v + q_1(x) v = (q_2-q_1)\partial_t u_{q_2} & \text{in } \Omega \times (0, T), \\
v(0) = 0, \quad \partial_t v(0) = (q_2-q_1)u_0 & \text{in } \Omega,\\
v = 0 & \text{on } \partial\Omega \times (0, T). \\
\end{cases}
\end{equation}
Using Lemma \ref{Car}, we obtain
\begin{equation}
s^{1/2}\int_\Omega e^{2s\varphi(x,0)}|(q_2-q_1)u_0|^2dx\leq M \int_{0}^T \int_\Omega e^{2s\varphi} |(q_2-q_1)\partial_t u_{q_2}|^2 dxdt + M s \int_{0}^T \int_{\Gamma_0} e^{2s\varphi} |\partial_\nu v|^2 d\sigma dt.
\end{equation}
Since
\begin{equation}
|u_0(x)| \geq m_0 > 0, \quad x \in \Omega \quad \text{and} \quad \|u_{q_2}\|_{H^1(0,T;L^\infty(\mathsf{M}))} \leq M_0,
\end{equation}
we have
\begin{equation}\label{qq}
s^{1/2}\int_\Omega e^{2s\varphi(x,0)}|q_2-q_1|^2dx\leq  M \int_{0}^T \int_\Omega e^{2s\varphi} |q_2-q_1|^2 dxdt + M s \int_{0}^T \int_{\Gamma_0} e^{2s\varphi} |\partial_\nu v|^2 d\sigma dt.
\end{equation}
Therefore
\begin{equation}
\begin{aligned}
 \int_{0}^T \int_\Omega e^{2s\varphi(x,t)} |q_2-q_1|^2 dxdt &=\int_\Omega \Big(\int_{0}^T e^{2s(\varphi(x,t)-\varphi(x,0))}dt\Big) e^{2s\varphi} |q_2-q_1|^2 dx\\
 &=\int_\Omega(\int_{0}^T \Big(e^{-2s[e^{\lambda(|x-x_0|^2+\beta_0)}(1-e^{-\lambda \beta t^2})]}dt\Big) e^{2s\varphi(x,0)} |q_2-q_1|^2 dx
\end{aligned}
\end{equation}
Let 
\[k(s)=\int_{0}^T e^{-2s[e^{\lambda(|x-x_0|^2+\beta_0)}(1-e^{-\lambda \beta t^2})]}dt,\]
easily obtain
\[ k(s)=o(1) \ \text{as } s\to\infty. \]
Thus, choosing \(s>0\) sufficiently large, we absorb the first term on the right-hand side
of (\ref{qq}) into the left-hand side:
\begin{equation}
s^{1/2}\int_\Omega |q_2-q_1|^2dx\leq C  \int_{0}^T \int_{\Gamma_0} e^{2s\varphi} |\partial_\nu \partial_t(u_{q_1}-u_{q_2})|^2 d\sigma dt.
\end{equation}
This proves Theorem \ref{Thm}.

\section*{Acknowledgment}
The work of Suliang Si is supported by  the Shandong Provincial Natural Science Foundation (No. ZR2022QA111).

\end{document}